\documentclass[10pt]{article}

\usepackage[english]{babel}
\usepackage{latexsym,amssymb,amsmath,amsthm,amsfonts}
\usepackage{float,graphicx}
\usepackage{adjustbox}
\usepackage{pgfplots}
\usepackage{url} 
\pgfplotsset{compat=1.18}
\usepackage[breaklinks=true]{hyperref}
\usepackage{algorithm}
\usepackage{algpseudocode}

\newtheorem{theorem}{Theorem}[section]

\theoremstyle{definition}

\theoremstyle{remark}

\numberwithin{equation}{section}

\title{Genetic algorithm and edge-colorings of complete graphs with connected classes}

\author{
Jorge Cervantes-Ojeda \footnotemark[1]
\and María C. Gómez-Fuentes \footnotemark[1]
\and Christian Rubio-Montiel \footnotemark[2]
 }

\begin{document}
\maketitle

\def\thefootnote{\fnsymbol{footnote}}
\footnotetext[1]{Universidad Autónoma Metropolitana, Unidad Cuajimalpa, Mexico, {\tt [jcervantes|mgomez]@cua.uam.mx}. ORCID: [0000-0002-2267-7165,0000-0003-0033-4476].}
\footnotetext[2]{División de Matemáticas e Ingeniería, FES Acatlán, Universidad Nacional Autónoma de México, Mexico,
{\tt christian.rubio@acatlan.unam.mx}. ORCID: 0000-0003-1474-8362.}

\begin{abstract}
In this study, the Rank Genetic Algorithm was adapted to address a problem in the field of Chromatic Graph Theory, namely, on the parameter called the connected-pseudoachromatic index. We successfully improved several previously known bounds of that index of the complete graph.
\end{abstract}



\section{Introduction}
\label{sec:in}
Genetic Algorithms (GA) have emerged as powerful tools, taking advantage of efficient strategies to address complex optimization problems prevalent in real-world applications in diverse industries, including transportation and manufacturing \cite{Gen&Lin2023}. Although GAs are widely recognized for their effectiveness in addressing practical engineering challenges, their interdisciplinary potential extends to the realm of theoretical advancements in mathematical research.

In our study, we explore the application of a particular GA: the Rank Genetic Algorithm (Rank GA) introduced in \cite{cervantes2008rank} to address a graph theory problem. The Rank GA incorporates genetic operators that facilitate both local search (exploitation) and global search (exploration) simultaneously. This unique feature allows for a delicate balance between exploration and exploitation. With dedicated individuals focusing on exploration, the algorithm can escape local optima. Simultaneously, during local searches, the genes of the best individuals are exploited to refine the overall solution, contributing to the algorithm's effectiveness. Due to its versatility, the Rank GA has proven to be highly effective in solving combinatorial optimization problems, as demonstrated in successful applications \cite{cervantes2023applying,MR3924131,MR4555460}.

In recent decades, GAs have emerged as effective tools to address problems of graph theory with practical applications. Specifically, various GAs have been developed to tackle the Graph Coloring Problem (GCP), a fundamental issue in Graph Theory. In the simplest instance of the GCP, the objective is to assign colors to the vertices of a graph in such a way that no two adjacent vertices share the same color. The goal is to find a valid coloring of the graph with the minimal number of colors. Given its NP-hard nature and numerous practical applications, the GCP stands out as the most extensively researched combinatorial optimization problem in both computer science and mathematics \cite{Douiri2015}. In \cite{MR4076285}, we can find the most significant work in which GAs have been used to address the GCP. In addition to GCP, GAs have also been applied to other graph problems, among the most recent works are: the Shortest Path Problem \cite{CosmaEtAl2021}, the Traveling Salesman Problem \cite{XuEtAl2022}, Maximum Matching in bipartite graphs \cite{Long2019}, and the Maximum Edge-Disjoint Path Problem \cite{Gen&Lin2023}.

We are aware of the following studies where a genetic algorithm played a crucial role in addressing theoretical challenges: In \cite{MR3924131} the Rank GA was used to find rainbow $t$-colorings of the family of Moore cages with girth 6: $(k; 6)$, achieving an improvement in the upper bound for the 4-rainbow coloring of a $(4; 6)$-cage. In addition, in \cite{MR4555460}, the Rank GA was adapted with recombination and mutation operators from graph theory. These operators enabled the discovery of Hajós operations to construct the symmetric 5-cycle from the complete symmetric digraph of order 3 using only Hajós operations. Finally, in \cite{cervantes2023applying} real solutions were identified by a genetic algorithm, thus contributing to the validation of a conjecture.

Our specific attention is directed towards the connected-pseudoachromatic index of the complete graph, denoted $\psi_c'(K_n)$ (for short $\psi_c(n)$). Determining $\psi_c(n)$ is difficult, and so far, there have been very few exact results for $\psi_c(n)$ of complete graphs with connected classes. However, it has been possible to bound it using analytical methods \cite{MR3695270}.

In this interdisciplinary study, our objective was to refine the known bounds of $\psi_c(n)$ exploiting the capabilities of the Rank GA. We successfully improved several of the previously known bounds. This synergy between Genetic Algorithms and mathematical theory contributes to the intersection of practical problem-solving and theoretical advancements.

The structure of this paper is as follows. In Section \ref{sec:problem_statement}, we state the problem and include the necessary definitions to understand it. In Section \ref{sec:rank_ga}, we describe the Rank GA and in Section \ref{sec:results}, we present the obtained results. Finally, Section \ref{sec:conclusion} contains the conclusions.

\section{Problem Statement}\label{sec:problem_statement}
The maximum number of colors in a connected and complete coloring of a complete graph $K_n$ of order $n$ that can be achieved is called the connected-pseudoachromatic index and is denoted $\psi_c(n)$. There are certain values of $n$ for which it is known, and for others, only the lower and upper bounds are known \cite{MR3265137,MR3695270}. The problem we address here, in particular, consists in improving the known bounds.

\subsection{Edge-Colorings of a Complete Graph}
A \emph{graph} is defined as a set of vertices $V$ and a set of edges $E$, and is denoted by $G=(V,E)$. It is important to note that, in a graph, there cannot be repeated edges or edges connecting a vertex to itself, since these conditions would violate the definition of a graph. A \emph{complete graph} is a graph in which each pair of vertices is connected by an edge.

A \emph{complete edge-coloring} in a complete graph is an assignment of colors to edges such that, for every pair of colors, there is at least a common vertex. Figure \ref{fig1} (left) shows an example of a complete edge-coloring of $K_4$, while Figure \ref{fig1} (right) shows an example of a noncomplete coloring, since the color pair (``dotted'', ``dashed'') does not have a common vertex.

\begin{figure}[htbp]
  \begin{center}
    \includegraphics{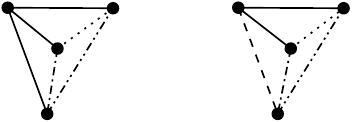}
    \caption{Edge-colorings in the complete graph $K_4$. (Left) A complete coloring with 4 colors. (Right) A noncomplete coloring with 5 colors.}
    \label{fig1}
  \end{center}
\end{figure}

A \emph{chromatic class} is the set of all edges of a colored graph that have the same color. The \emph{induced subgraph} of a chromatic class $X$ is formed by the edges of the chromatic class $X$ and the vertices connected to those edges. 

A graph is \emph{connected} when there is at least one path between each pair of its vertices. A chromatic class is \emph{connected} if its induced subgraph is connected. Each chromatic class in Figure \ref{fig1} is connected. A coloring is \emph{connected} if all its chromatic classes are connected chromatic classes.

The \emph{connected-pseudoachromatic index} $\psi'_c(G)$ of a graph $G$ is the maximum number of colors $k$ for which a connected and complete edge-coloring exists in the graph $G$ using $k$ colors, when $G$ is connected; otherwise, it is defined as the maximum index over its connected components. When we color vertices and each chromatic class is connected, the parameter is also known as the Hadwiger number of $G$ from the perspective of minor graphs.

When a graph is complete, the connected-pseudoachromatic index is denoted by $\psi_c(n)$ where $n$ is the number of vertices in the complete graph $K_n$. 

In \cite{MR3695270} it is shown that $\psi_c(n)$ can be bounded by analytical methods obtaining that $\psi_c(n)=\Theta(n^{3/2})$. Table \ref{table1} shows the known bounds of $\psi_c(n)$, for complete and connected graphs of $2\leq n\leq 31$ vertices. With the use of the computer, complete colorings can be achieved, and thus improve the lower bound of $\psi_c(n)$.

\begin{table}[htbp]
\begin{adjustbox}{center}
\begin{tabular}{|c|cccccccccc|}
\hline
     $n$&2&3&4&5&6&7&8&9&10&11  \\
\hline
 Upper& \emph{1}& \emph{3}& \emph{4}& \emph{6}& \emph{7}& \emph{10}& 14& 18& 22&25\\
 Lower& \emph{1}& \emph{3}& \emph{4}& \emph{6}& \emph{7}& \emph{10}& 11& 12& 13&14\\
\hline
\hline
     $n$&12&13&14&15&16&17&18&19&20&21\\
\hline
 Upper& 28& 31& 34& 37& 40& 45& 51& 57& 63&70\\
 Lower& 19& 26& 27& 28& 29& 30& 31& 32& 33&42\\
\hline
\hline
     $n$&22&23&24&25&26&27&28&29&30&31\\
\hline
 Upper& 74& 78& 82& 86& 90& 94& 98& 102& 106&116\\
 Lower& 43& 44& 45& 46& 47& 48& 49& 50& 51&93\\
\hline
\end{tabular}
\end{adjustbox}
\caption{Values for $2\protect\leq n\protect\leq 7$ and lower bounds for $8\protect\leq n\protect\leq 12$ given in \cite{MR3265137}. The other values were given in \cite{MR3695270}.}\label{table1}
\end{table}

\section{Description of the Rank GA}\label{sec:rank_ga}
Genetic algorithms draw inspiration from the genetic evolution process observed in living organisms, addressing challenges in combinatorial optimization, network design, routing, scheduling, location and allocation, reliability design, and logistics \cite{Gen&Lin2023}. Leveraging this evolutionary concept, we use the Rank GA to solve the problem described in Section \ref{sec:problem_statement}, making use of its ability to escape local optima and refine solutions. The Rank GA operates by evaluating and ranking the population before applying each genetic operator. This ranking ensures that the corresponding genetic operator is uniquely applied to each individual based on their position in the ordered population. This approach enhances the ability to search for optimal solutions and to improve the bounds obtained thus far.

Individuals recombine with others of similar rank, so if one individual is fit due to possessing certain beneficial genes and another individual is fit for having different advantageous genes, their recombination can yield individuals that combine both sets of genes. If two individuals are unfit, they are likely to be quite different because there are numerous ways to be unfit. When recombined, it is highly probable that the offspring will also be distinct, exploring areas distant from the parents and, thereby, increasing the likelihood of escaping local optima. The fittest individuals tend to propagate their genes more within the population than the less fit ones. This allows the concentration of the population around the best individual, guiding the population towards the best individuals by exploiting their genes more.

In summary, with the Rank GA the worst individuals are dedicated to exploration making possible to escape from the local optima and, at the same time, the best individuals perform local search exploiting their genes to refine the best solution obtained so far. Therefore, the Rank GA is particularly suitable for optimization problems where finding a global optimum is challenging due to complex search spaces, as in the problem that we address in this study.

\subsection{The Genetic Operators of the Rank GA}
The Rank GA follows a specific process, starting with random initialization of the population. Subsequently, the individuals are evaluated and ranked from the best to the worst based on their fitness. The rank of the \(i\)-th individual in the ordered population, denoted as \(r_i\), is determined using equation \eqref{eq:rank}.

\begin{equation}
\ r_i = \frac{i}{N-1} \  
\label{eq:rank}
\end{equation}

where \(i\) ranges from 0 to \(N-1\), and \(N\) represents the number of individuals in the population.

Once individuals are ranked based on their fitness, from best to worst, the Rank GA applies the following operations iteratively until a stopping criterion is met: \textit{Rank-based selection}, \textit{Sort and rank}, \textit{Rank-based recombination}, \textit{Evaluate, sort, and rank}, \textit{Rank-based mutation} and \textit{Evaluate, sort, and rank}.

\subsubsection{Rank-Based Selection}
\textit{Rank-based selection} involves the cloning of individuals according to a two-step procedure. The desired number of clones, \texttt{cloneNbr}, is calculated for each individual $i$ using Equation \eqref{eq:selection}:

\begin{equation}
\ \text{cloneNbr}_i = S \cdot (1 - r_i)^{S-1}  
\label{eq:selection}
\end{equation}

Here, $r_i$ represents the rank of individual $i$, ranging from 0 to 1, and $S$ corresponds to the selective pressure of the rank-based selection operator. In the first step, the floor of $cloneNbr_i$ is taken, resulting in the number of clones to generate for each individual $i$. In the second step, additional clones are produced as follows: the procedure starts with $i=0$ and, while the total number of clones is less than the original number of individuals $N$, the fractional part of $cloneNbr_i$ is used as the probability of producing an additional individual clone $i$. If a random number in the range [0,1) is less than that probability, a clone of individual $i$ is generated. The value of $i$ is increased by modulo $N$ to ensure that it passes through all individuals.

\subsubsection{Rank-Based Recombination}
The \textit{Rank-Based Recombination} operator forces mating between individuals with indices $i$ and $i$+1 in the ordered population, where $i$ increases by 2. This strategy ensures that individuals mate with others who are close in terms of rank. Using \textit{ rank-based recombination} with a selective pressure value of $S=3$, the Rank GA achieves a balance between preserving the advantageous genes of the fittest individual and introducing local search through recombination with other fit individuals.

\subsubsection{Mutation by Rank}
The mutation probability assigned to each individual, denoted as $p_i$, is determined by a monotonically increasing function of the rank. The function is defined by Equation \eqref{eq:mutation} as follows:

\begin{equation}
p_i = p_{\text{max}} \cdot r_i^{(\frac{\ln(p_{\text{max}} \cdot G)}{\ln(N-1)})}  \label{eq:mutation}
\end{equation}

In this equation, $p_{\text{max}}$ represents the maximum mutation probability that can be assigned, $r_i$ is the rank of individual $i$, $G$ is the genotype size, and $N$ is the population size. The function assigns a mutation probability of 0 to the best individual ($r_0=0$), $1/G$ to the second-best individual, and $p_{\text{max}}$ to the worst individual.

\subsection{Pseudocode}

The pseudocode for the Rank Genetic Algorithm (Rank GA), outlining the key steps involved in its execution, is shown in Algorithm \ref{alg:rank_ga} on page \pageref{alg:rank_ga}.

\begin{algorithm}
\caption{Rank Genetic Algorithm (Rank GA)}
\label{alg:rank_ga}
\begin{algorithmic}
\State \textbf{STEP 1: Initialization}
\State Initialize population \(\mathcal{P}\) randomly with \(N\) individuals
\State Evaluate the fitness of each individual in \(\mathcal{P}\)

\While{stopping criterion not met}
    \State \textbf{STEP 2: Rank-Based Selection}
    \State Sort \(\mathcal{P}\) by fitness in descending order
    \State Compute \(r_i\) for each individual \(i\) using equation \eqref{eq:rank}
    
    \State \textbf{STEP 2.A: Cloning Based on Integer Part of Clone Number}
    \State Initialize an empty list \(\mathcal{P}'\) to store clones
    \For{each individual \(i\) in \(\mathcal{P}\)}
        \State Calculate \texttt{cloneNbr\(_i\)} using equation \eqref{eq:selection}
        \State Generate floor(\texttt{cloneNbr\(_i\)}) clones of individual \(i\) and add them to \(\mathcal{P}'\)
    \EndFor
    
    \State \textbf{STEP 2.B: Fractional Cloning (Cyclic)}
    \State Initialize \(i = 0\)
    \While{size of \(\mathcal{P}'\) is less than \(N\)}
        \State Let \(f_i\) be the fractional part of \texttt{cloneNbr\(_i\)}
        \If{random number $<$ \(f_i\)}
            \State Generate one additional clone of individual \(i\) and add it to \(\mathcal{P}'\)
        \EndIf
        \State Increment \(i\) by 1 (modulo \(N\)) to cycle through the population
    \EndWhile
    
    \State Replace original population \(\mathcal{P}\) with the cloned population \(\mathcal{P}'\)
    
    \State \textbf{STEP 3: Rank-Based Recombination}
    \State Sort \(\mathcal{P}\) by fitness in descending order
    \State Compute \(r_i\) for each individual \(i\) using equation \eqref{eq:rank}
    \For{each pair of individuals \(i\) and \(i+1\) in \(\mathcal{P}\)}
        \State Perform recombination between individuals \(i\) and \(i+1\)
        \State Replace the parents \(i\) and \(i+1\) with the two offspring
    \EndFor
    \State Evaluate the fitness of each individual in \(\mathcal{P}\)
    
    \State \textbf{STEP 4: Mutation by Rank}
    \State Sort \(\mathcal{P}\) by fitness in descending order
    \State Compute \(r_i\) for each individual \(i\) using equation \eqref{eq:rank}
    \For{each individual \(i\) in \(\mathcal{P}\)}
        \State Calculate mutation probability \(p_i\) using equation \eqref{eq:mutation}
        \For{each gene \(g\) in the genotype of individual \(i\)}
            \If{random number $<$ \(p_i\)}
                \State Mutate gene \(g\)
            \EndIf
        \EndFor
    \EndFor
    \State Evaluate the fitness of each individual in \(\mathcal{P}\)
    
\EndWhile

\State \textbf{STEP 5: Return the Best Individual}
\State \textbf{Return} the best individual in \(\mathcal{P}\)
\end{algorithmic}
\end{algorithm}

\section{Adaptation of the RankGA to the Edge Coloring Problem}\label{sec:rankga_adaptation}

To address the problem of improving the bounds of the connected-pseudoachromatic index $\psi_c(n)$ of a complete graph $K_n$, we adapted the \textbf{RankGA} (Ranking Genetic Algorithm). This section describes the representation, fitness function, and mutation mechanism used in this adaptation.

\subsection{Representation}

Each solution, or individual, in the genetic algorithm is represented as an array of integers, where each integer corresponds to the color assigned to a specific edge in $K_n$. The range of possible values for each gene is from $0$ to $(\text{numColors} - 1)$, representing the available palette of colors. The length of the array equals the total number of edges in the complete graph, $\binom{n}{2}$. This representation encodes a coloring of the graph that the algorithm evaluates and evolves toward better solutions.

\subsection{Fitness Function}

The fitness function is designed to evaluate the quality of an edge coloring with respect to the constraints and objectives of the problem. It incorporates multiple criteria:

\begin{itemize}
    \item \textbf{Number of Colors Used}: The function counts the distinct colors used in the coloring.
    \item \textbf{Vertex Connections}: For any two edges sharing a vertex, their colors must differ. This is enforced by tracking the presence of each color at every vertex.
    \item \textbf{Pairwise Connectivity}: Penalties are applied if any two used colors do not share at least one vertex in their respective edges.
    \item \textbf{Disjoint Color Classes}: Colors forming disconnected components are penalized.
    \item \textbf{Color Distribution}: The standard deviation and average of the edge counts per color are computed to encourage uniform distribution of edges across colors.
\end{itemize}

The fitness function is calculated as:
\[
\text{fitness} = \alpha - (\beta \cdot \text{weightPairs}) - (\gamma \cdot \text{weightColors}) - (\text{std} \cdot \text{weightStd}) - (\text{avg} \cdot \text{weightAvg}),
\]
where:
\begin{itemize}
    \item $\alpha$ is the number of distinct colors used.
    \item $\beta$ penalizes unconnected color pairs.
    \item $\gamma$ penalizes disjoint components.
    \item $\text{std}$ and $\text{avg}$ account for uneven color distribution.
    \item $\text{weightPairs}$, $\text{weightColors}$, $\text{weightStd}$, and $\text{weightAvg}$ are weighting factors.
\end{itemize}

A solution satisfying all constraints and using the maximum number of colors achieves a fitness value equal to $\alpha$.

\subsection{Mutation}

The mutation operator introduces diversity into the population by altering the color of a randomly selected edge. The new color is chosen randomly from the available palette, ensuring that the solution remains valid. This mechanism allows the algorithm to explore new regions of the search space while maintaining feasibility.

\section{Results}\label{sec:results}

\subsection{On the lower bound}
Table \ref{table2} presents the improved lower bounds obtained through the Rank GA for various values of $n$. The known upper bounds for $\psi_c(n)$ are also shown for reference in Table \ref{table2}. By comparing the improved lower bounds with the existing lower and upper bounds, we observed significant reductions in the gap between them.

\begin{table}[htbp]
\begin{adjustbox}{center}
\begin{tabular}{|c|cccccccccc|}
\hline 
$n$ & $2$ & $3$ & $4$ & $5$ & $6$ & $7$ & $8$ & $9$ & $10$ & $11$\\
\hline 
Upper & \emph{1} & \emph{3} & \emph{4} & \emph{6} & \emph{7} & \emph{10} & $14$ & $18$ & $22$ & $25$\\
Lower Improved &  &  &  &  &  &  &  & $13$ & $16$ & $18$\\
Lower & \emph{1} & \emph{3} & \emph{4} & \emph{6} & \emph{7} & \emph{10} & $11$ &  &  & \\
\hline
\hline 
$n$ & $12$ & $13$ & $14$ & $15$ & $16$ & $17$ & $18$ & $19$ & $20$ & $21$\\
\hline 
Upper & $28$ & $31$ & $34$ & $37$ & $40$ & $45$ & $51$ & $57$ & $63$ & $70$\\
Lower Improved & $21$ &  &  &  &  & $32$ & $35$ & $37$ & $39$ & $43$\\
Lower &  & $26$ & $27$ & $28$ & $29$ &  &  &  &  & \\
\hline
\hline 
$n$ & $22$ & $23$ & $24$ & $25$ & $26$ & $27$ & $28$ & $29$ & $30$ & $31$\\
\hline 
Upper & $74$ & $78$ & $82$ & $86$ & $90$ & $94$ & $98$ & $102$ & $106$ & $116$\\
Lower Improved & $46$ & $50$ & $52$ & $54$ & $55$ & $59$ & $62$ & $64$ & $66$ & \\
Lower &  &  &  &  &  &  &  &  &  & $93$\\
\hline 
\end{tabular}
\\\end{adjustbox}
\caption{Improved lower bounds with the Rank GA for $\psi_{c}(n)$ and $n\leq 31$.}\label{table2}
\end{table}

Figure \ref{fig:plot} visually represents the values detailed in Table \ref{table2}. As can be seen, the known lower bound presents sudden steps at certain graph sizes, while the improved lower bounds present smoother growth.

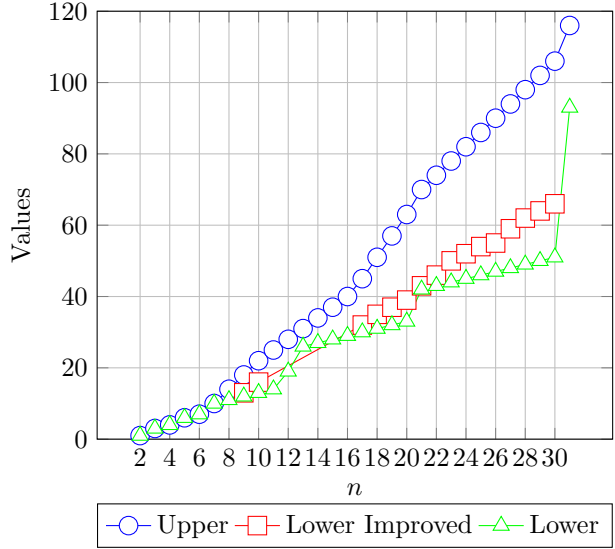
\begin{figure}[htbp]
    \centering
    \begin{adjustbox}{center}
        \begin{tikzpicture}
            \begin{axis}[
                xlabel={$n$},
                ylabel={Values},
                legend style={at={(0.5,-0.15)}, anchor=north, legend columns=-1},
                symbolic x coords={2,3,4,5,6,7,8,9,10,11,12,13,14,15,16,17,18,19,20,21,22,23,24,25,26,27,28,29,30,31},
                xtick={2,4,6,8,10,12,14,16,18,20,22,24,26,28,30},
                ymin=0,
                ymax=120,
                grid=major,
                mark size=3.5,
                mark options={solid, fill=white}
            ]

            \addplot[blue, mark=*] coordinates {
                (2,1) (3,3) (4,4) (5,6) (6,7) (7,10) (8,14) (9,18) (10,22) (11,25) (12,28) (13,31) (14,34) (15,37) (16,40) (17,45) (18,51) (19,57) (20,63) (21,70) (22,74) (23,78) (24,82) (25,86) (26,90) (27,94) (28,98) (29,102) (30,106) (31,116)
            };
            \addlegendentry{Upper};

            \addplot[red, mark=square*] coordinates {
                (9,13) (10,16) (17,32) (18,35) (19,37) (20,39) (21,43) (22,46) (23,50) (24,52) 
                (25,54) (26,55) (27,59) (28,62) (29,64) (30,66)
            };
            \addlegendentry{Lower Improved};

            \addplot[green, mark=triangle*] coordinates {
                (2,1) (3,3) (4,4) (5,6) (6,7) (7,10) (8,11) (9,12) (10,13) (11,14) (12,19) (13,26) (14,27) (15,28) (16,29) (17,30) (18,31) (19,32) (20,33) (21,42) (22,43) (23,44) (24,45) (25,46) (26,47) (27,48) (28,49) (29,50) (30,51) (31,93)
            };
            \addlegendentry{Lower};

            \end{axis}
        \end{tikzpicture}
    \end{adjustbox}
    \caption{Improved lower bounds with the Rank GA for $\psi_{c}(n)$.}
    \label{fig:plot}
\end{figure}

\subsection{On the upper bound}\label{sec:upper_bound}
In \cite{MR3695270} it is shown that a complete graph of $n$ vertices accepts a coloring with $(\sqrt{n-1})^3/2$ colors and that a coloring with $(n-1)(\sqrt{n/2 + 1/16} + 1/4)$ colors or more cannot be complete and connected at the same time.

On the one hand, the asymptotic growth of $\psi_c(n)$ is $\Theta(n^{\frac{3}{2}})$, on the other hand, the coefficient of the main term of the lower bound is $\frac{1}{2}$, while it is $\frac{1}{\sqrt{2}}$ for the upper bound. This gives a significant difference in the values in Table \ref{table1}.

\begin{theorem}\label{teo1}
If $n\geq8$, then an approximate upper bound is \[\psi_{c}(n) \leq k\approx \frac{n(n-1)}{\sqrt{4n-3}-1} = \frac{1}{2}n^{\frac{3}{2}}+O(n),\]
\end{theorem}

\begin{proof}
Let $n\geq 8$ and  $\varsigma\colon E(K_n) \rightarrow \{1,\dots,k\}$ a connected and complete edge-coloring of $K_n$ with $k=\psi_c(n)$ colors.

Let $x$ be the cardinality of the smallest chromatic class of $\varsigma$, that is, let $x= min\{ \left|\varsigma^{-1}(i)\right| : i\in \{1,\dots,k\} \}$. Without loss of generality, suppose that $x= \left|\varsigma^{-1}(k)\right|$.

On the one hand, $\varsigma$ defines a partition of $E(K_n)$ and it follows that $k\leq f_n(x):=n(n-1)/(2x)$.

On the other hand, $\varsigma$ is connected, therefore, $\varsigma^{-1}(k)$ is at least a tree subgraph. The maximum number of spanning trees in the subgraph $H=K_{x+1}$, such that $V(H)=V(\varsigma^{-1}(k))$, is at most $\frac{\binom{x+1}{2}}{x}$, then the chromatic class $x$ is incident to at most $\frac{\binom{x+1}{2}}{x}-1=\frac{x-1}{2}$ other chromatic classes. In addition, there are $(x+1)(n-(x+1))$ edges that are incident to some vertex of $\varsigma^{-1}(k)$ at least once, but on average each class is $\frac{2x}{x+1}$ incident to $\varsigma^{-1}(k)$. Since $\varsigma$ is complete, the number of chromatic classes incident to $\varsigma^{-1}(k)$ containing no edge in $H$ is at most 
\begin{equation}\label{eq5}
\frac{(x+1)(n-x-1)}{2x/(x+1)}
\end{equation}
in average.

Hence, there are at most $g_n(x)-1$ chromatic classes incident with some edge in $\varsigma^{-1}(k)$ where $g_n(x)-1:=\frac{(x+1)^2}{2x}(n-(x+1))+\frac{x-1}{2}$, according to the hypothesis of the average degree, i.e.
\[g_n(x)=\frac{n(x^2+2x+1)-x^3-2x^2-2x-1}{2x}.\]

Then, we have $\psi_{c}(n)\leq k\approx \min\{f_n(x),g_n(x)\}$ and then
\[\psi_{c}(n)\leq k\approx \max\left\{ \min\{f_n(x),g_n(x)\} \textrm{ with } x\in\mathbb{N}\right\}.\]

As $f_n$ is a hyperbola and $g_n$ is a parabola there are two positive values $x_0$ and $x_1$ such that $x_0<x_1$, $g_n(x_0)=f_n(x_0)$ and $g_n(x_1)=f_n(x_1)$ for which $f_n(x) \leq g_n(x)$ for $x_0 \leq x \leq x_1$ and $g_n(x)<f_n(x)$ in any other case when $x \in \mathbb{R}^{+}$. Since $f_n(x_0)>f_n(x_1)$, it follows that $g_n(x_0)=f_n(x_0)=\max\left\{ \min\{f_n(x),g_n(x)\} \textrm{ with } x\in\mathbb{R}^{+}\right\}$, see Figure \ref{fig2}.

\begin{figure} [htbp]
\begin{center}
\includegraphics{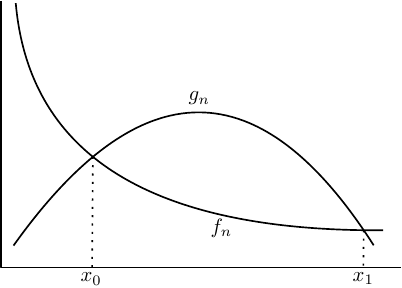}
\caption{The functions $g_n$ and $f_n$ for some fixed value $n$.}\label{fig2}
\end{center}
\end{figure}

Equation $f_n(x_0)=g_n(x_0)$ is reduced to $n^2-(x_0^2+2x_0+2)n+x_0^3+2x_0^2+2x_0+1=0$. For the positive solution, we obtain $n=x_0^2+x_0+1$ and then $x_0=\frac{\sqrt{4n-3}-1}{2}$ taking the positive solution. 

Since $f_n(x_0)=\frac{n(n-1)}{2x_0}$ we get $f_n(x_0)=\frac{n(n-1)}{\sqrt{4n-3}-1},$ therefore \[\psi_c(n)\leq k\approx \frac{n(n-1)}{\sqrt{4n-3}-1}\]
and the result is established.
\end{proof}

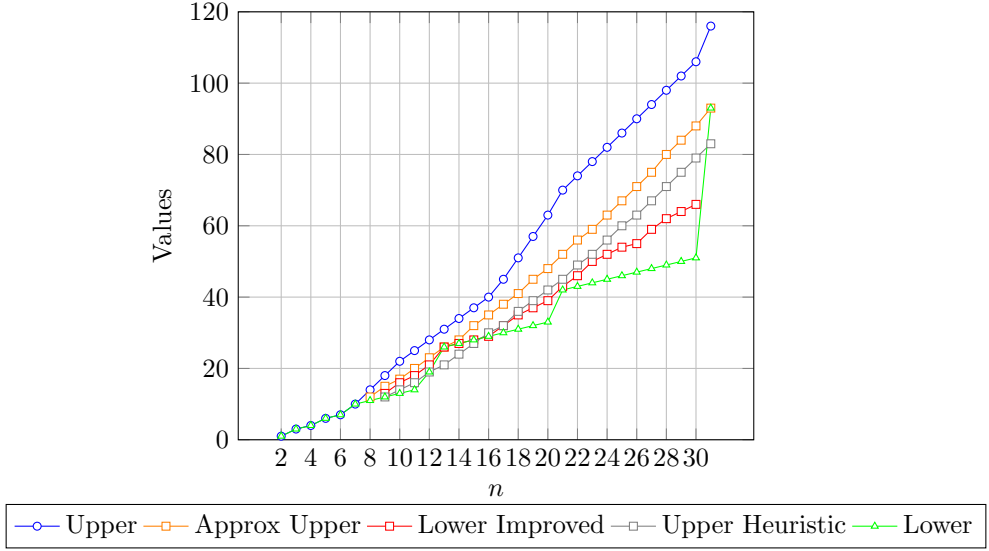
\begin{figure}[htbp]
    \centering
    \begin{adjustbox}{center}
        \begin{tikzpicture}
            \begin{axis}[
                xlabel={$n$},
                ylabel={Values},
                legend style={at={(0.5,-0.15)}, anchor=north, legend columns=-1},
                symbolic x coords={2,3,4,5,6,7,8,9,10,11,12,13,14,15,16,17,18,19,20,21,22,23,24,25,26,27,28,29,30,31},
                xtick={2,4,6,8,10,12,14,16,18,20,22,24,26,28,30},
                ymin=0,
                ymax=120,
                grid=major,
                mark size=1.5,
                mark options={solid, fill=white}
              ]

            \addplot[blue, mark=*] coordinates {
                (2,1) (3,3) (4,4) (5,6) (6,7) (7,10) (8,14) (9,18) (10,22) (11,25) (12,28) (13,31) (14,34) (15,37) (16,40) (17,45) (18,51) (19,57) (20,63) (21,70) (22,74) (23,78) (24,82) (25,86) (26,90) (27,94) (28,98) (29,102) (30,106) (31,116)
            };
            \addlegendentry{Upper};

            \addplot[orange, mark=square*] coordinates {
                (8,12) (9,15) (10,17) (11,20) (12,23) (13,26) (14,28) (15,32) (16,35) (17,38) (18,41) (19,45) (20,48) (21,52) (22,56) (23,59) (24,63) (25,67) (26,71) (27,75) (28,80) (29,84) (30,88) (31,93)
            };
            \addlegendentry{Approx Upper};

            \addplot[red, mark=square*] coordinates {
                (9,13) (10,16) (11,18) (12,21)
                (13,26) (14,27) (15,28) (16,29)
                (17,32) (18,35) (19,37) (20,39) (21,43) (22,46) (23,50) (24,52) (25,54) (26,55) (27,59) (28,62) (29,64) (30,66)
            };
            \addlegendentry{Lower Improved};

            \addplot[green, mark=triangle*] coordinates {
                (2,1) (3,3) (4,4) (5,6) (6,7) (7,10) (8,11) (9,12) (10,13) (11,14) (12,19) (13,26) (14,27) (15,28) (16,29) (17,30) (18,31) (19,32) (20,33) (21,42) (22,43) (23,44) (24,45) (25,46) (26,47) (27,48) (28,49) (29,50) (30,51) (31,93)
            };
            \addlegendentry{Lower};

            \end{axis}
        \end{tikzpicture}
    \end{adjustbox}
    \caption{Improved bounds for $\psi_{c}(n)$.}
    \label{fig:plot2}
\end{figure}

\begin{table}[htbp]
\begin{adjustbox}{center}
\begin{tabular}{|c|cccccccccc|}
\hline 
$n$ & $2$ & $3$ & $4$ & $5$ & $6$ & $7$ & $8$ & $9$ & $10$ & $11$\\
\hline 
Upper & \emph{1} & \emph{3} & \emph{4} & \emph{6} & \emph{7} & \emph{10} &  &  &  & \\
Approx Upper &  &  &  &  &  &  & $12$ & $15$ & $17$ & $20$\\
Lower Improved &  &  &  &  &  &  &  & $13$ & $16$ & $18$\\
Lower & \emph{1} & \emph{3} & \emph{4} & \emph{6} & \emph{7} & \emph{10} & $11$ &  &  & \\
\hline 
\hline 
$n$ & $12$ & $13$ & $14$ & $15$ & $16$ & $17$ & $18$ & $19$ & $20$ & $21$\\
\hline
Approx Upper & $23$ & $\emph{26}$ & $28$ & $32$ & $35$ & $38$ & $41$ & $45$ & $48$ & $52$\\
Lower Improved & $21$ &  &  &  &  & $32$ & $35$ & $37$ & $39$ & $43$\\
Lower &  & $\emph{26}$ & $27$ & $28$ & $29$ &  &  &  &  & \\
\hline 
\hline 
$n$ & $22$ & $23$ & $24$ & $25$ & $26$ & $27$ & $28$ & $29$ & $30$ & $31$\\
\hline 
Approx Upper & $56$ & $59$ & $63$ & $67$ & $71$ & $75$ & $80$ & $84$ & $88$ & \emph{93}\\
Lower Improved & $46$ & $50$ & $52$ & $54$ & $55$ & $59$ & $62$ & $64$ & $66$ & \\
Lower &  &  &  &  &  &  &  &  &  & \emph{93}\\
\hline 
\end{tabular}
\end{adjustbox}
\caption{Improved bounds for $\psi_{c}(n)$.}\label{table3}
\end{table}

As we can see in Table \ref{table3}, for the values of $n=13$ and $31$ there are good approximations, which is due to the existence of the finite projective plane of odd order $q=3$ and $5$, which gives a lower bound of $\psi_c(n)$ for $n=q^2+q+1$ when $q$ is an odd prime power according to the following result.

\begin{theorem}\cite{MR3695270} \label{teo2}
If $q$ is an odd prime power and $n=q^2+q+1$ then \[\left\lceil \frac{q}{2}\right\rceil n \leq \psi(n).\]
\end{theorem}

Our approximation $\frac{n(n-1)}{\sqrt{4n-3}-1}$ implies that $\left\lceil \frac{q}{2}\right\rceil n$ is a good approximation because $q$ is odd and then

$\frac{n(n-1)}{\sqrt{4n-3}-1}=\frac{n(q^2+q)}{\sqrt{4(q^2+q+1)-3}-1}=\frac{q(q+1)n}{\sqrt{4q^2+4q+1}-1}=\frac{q(q+1)n}{2q}=\frac{(q+1)}{2}n=\left\lceil \frac{q}{2}\right\rceil n.$

\section{Conclusions}\label{sec:conclusion}
The problem at hand involves improving, by adapting a genetic algorithm called Rank GA, the known bounds of the connected-pseudoachromatic index for the complete graph of order $n$, denoted by $\psi_c(n)$, which represents the maximum number of colors in a connected and complete coloring of a complete graph $K_n$ of order $n$. The Rank GA demonstrated significant effectiveness in improving $\psi_c(n)$ bounds for complete graphs with connected classes. 
Our paper enters into the set of problems with results for small $n$ values such that: \cite{MR3217845} where it was proved that the well-known Erd{\H o}-Faber-L{\' o}vasz conjecture is true for $n \leq 12$; \cite{MR1162526}
where the extremal values of the $C_4$-free graphs are shown for $n\leq 31$. While in \cite{MR2661401} all the 
Steiner Triplets Systems of order 19 were found.

The algorithm was able to find the solutions due to its balance in exploring the solution space and exploiting promising solutions. The improved bounds of $\psi_c(n)$ obtained by the Rank GA were validated against an analytical approximation, confirming their validity. The success of the Rank GA in enhancing the bounds of $\psi_c(n)$ may significantly contribute to theoretical advancements in Chromatic Graph Theory. Upon examination of the resulting colorings, it is observed that the sizes of the chromatic classes vary, with most sizes equal to the minimum size stipulated by Theorem \ref{teo1}. However, their distribution does not align with the color patterns described in Theorem \ref{teo2}, meaning that they lack a projective plane structure in their arrangement.

Beyond the specific problem of graph coloring, the findings underscore the interdisciplinary impact of utilizing genetic algorithms in theoretical mathematical research.

Finally, Figure \ref{fig:example_solution} presents a sample solution with $n=12$ and 21 colors. Additional examples and a Python-based visualizer can be downloaded from \footnote{\url{https://www.dropbox.com/scl/fo/ibudsyzgwj0gz8gzs8ifi/AEXSShINNc23KMsURaxS9w4?rlkey=jsgye9loui8gapqnbm89c0csq&dl=0}}. The visualizer allows users to toggle the visualization of each color class to analyze solutions. 

\begin{figure}[ht]
  \centering
  \includegraphics[width=0.8\textwidth]{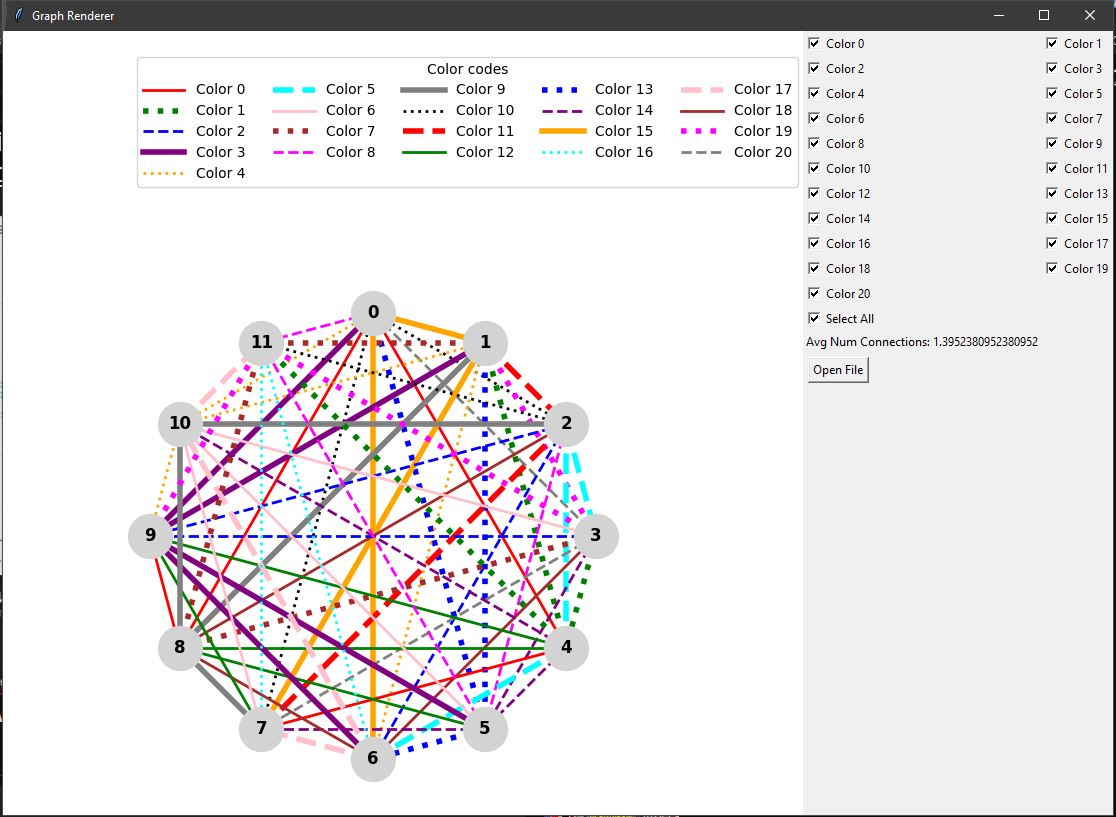}
  \caption{Example solution.}
  \label{fig:example_solution}
\end{figure}

\section*{Statements and Declarations}

The authors declare that no funds, grants, or other support were received during the preparation of this manuscript.

The authors have no relevant financial or non-financial interests to disclose.

All authors contributed to the study conception and design. The first draft of the manuscript was written by Jorge Cervantes-Ojeda, María C. Gómez-Fuentes and Christian Rubio-Montiel and all authors commented on previous versions of the manuscript. All authors read and approved the final manuscript.

\bibliographystyle{amsplain}
\bibliography{biblio}
\label{sec:biblio}

\end{document}